\documentclass[A4]{article}
\usepackage{amscd}
\usepackage{amssymb}
\usepackage{amsmath}
\usepackage{amsthm}
\usepackage{latexsym}
\usepackage{upref}
\usepackage{epsfig}
\usepackage{mathrsfs}
\usepackage{float}
\usepackage{indentfirst}
\usepackage{hyperref}
\hypersetup{
    colorlinks=true,
    linkcolor=blue,
    citecolor=blue,
    filecolor=blue,
    urlcolor=blue,
}

\usepackage{graphics,graphicx,color}
\usepackage{floatflt}
\usepackage{cite}
\usepackage{enumitem}
\setlist{font=\normalfont} 
\usepackage{fancyhdr} 
\usepackage[T1]{fontenc}
\usepackage{mathptmx} 

\theoremstyle{plain}
\newtheorem{theorem}{Theorem}[section]
\newtheorem{lemma}[theorem]{Lemma}
\newtheorem{proposition}[theorem]{Proposition}

\newtheorem{corollary}[theorem]{Corollary}
\theoremstyle{definition}
\newtheorem{definition}[theorem]{Definition}
\newtheorem{example}[theorem]{Example}
\theoremstyle{remark}

\numberwithin{equation}{section}

\makeatletter
\def\th@plain{%
  \thm@notefont{}
  \itshape 
}
\def\th@definition{%
  \thm@notefont{}
  \normalfont 
} \makeatother

\setlist{font=\normalfont}

\newcommand{\F}{\mathbb{F}}
\newcommand{\set}[1]{\{#1\}}
\newcommand{\cset}[2]{\set{{#1}\colon{#2}}}
\newcommand{\Fix}[1]{\mathrm{Fix}\,{#1}}
\newcommand{\LG}[2]{L^{#1}({#2})}
\newcommand{\id}[1]{\mathrm{id}\,_{#1}}
\newcommand{\sym}[1]{\mathrm{Sym}\,{(#1)}}
\newcommand{\C}{\mathbb{C}}
\newcommand{\gen}[1]{\langle#1\rangle}
\newcommand{\Z}{\mathbb{Z}}
\newcommand{\gyrL}[1]{L^\mathrm{gyr}(#1)}
\newcommand{\gyr}[2]{{\mathrm{gyr}[{#1}]}{#2}}
\newcommand{\orb}[1]{\mathrm{orb}\,{#1}}
\newcommand{\abs}[1]{|#1|}

\newcommand{\spn}[1]{\mathrm{span}\,{#1}}
\newcommand{\qt}[1]{``#1''}
\newcommand{\lsum}[2]{\displaystyle\sum_{#1}^{#2}}
\newcommand{\aorb}[2]{\mathrm{orb}_{#1}\,{#2}}
\newcommand{\stab}[1]{\mathrm{stab}\,{#1}}
\newcommand{\Cset}[2]{\left\{{#1}\colon{#2}\right\}}
\newcommand{\Gl}[1]{\mathrm{GL}({#1})}
\newcommand{\aut}[1]{\mathrm{Aut}\,{(#1)}}
\newcommand{\lbar}{\overline}
\newcommand{\res}[2]{{#1}\big|_{{#2}}}
\newcommand{\Bp}[1]{\left(#1\right)}
\newcommand{\operp}{\,{\footnotesize{\textcircled{$\perp$}}}\,}
\newcommand{\norm}[1]{\|#1\|}
\newcommand{\lhat}[1]{\widehat{#1}}
\newcommand{\Abs}[1]{\left|#1\right|}
\newcommand{\Res}[3]{\mathrm{Res}^{#1}_{#2}{#3}}
\newcommand{\Ind}[3]{\mathrm{Ind}^{#1}_{#2}{#3}}

\DeclareMathAlphabet{\cols}{OMS}{cmsy}{m}{n} %

\begin{document}
\title{\textbf{Frobenius reciprocity on the space of functions invariant under a group action}}
\author{Teerapong Suksumran\footnote{Corresponding author.}\,\,\,\href{https://orcid.org/0000-0002-1239-5586}{\includegraphics[scale=1]{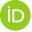}} \quad and\quad Tanakorn Udomworarat\\[5pt]
Research Center in Mathematics and Applied Mathematics\\
Department of Mathematics\\
Faculty of Science, Chiang Mai University\\
Chiang Mai 50200, Thailand\\[5pt]
\texttt{teerapong.suksumran@cmu.ac.th} (T. Suksumran)\\
\texttt{tanakorn\_u@cmu.ac.th} (T. Udomworarat)}
\date{}
\maketitle

\begin{abstract}
This article studies connections between group actions and their corresponding vector spaces. Given an action of a group $G$ on a nonempty set $X$, we examine the space $L(X)$ of scalar-valued functions on $X$ and its fixed subspace:
$$
L^G(X) = \{f\in L(X)\colon f(a\cdot x) = f(x) \textrm{ for all }a\in G, x\in X\}.
$$
In particular, we show that $L^G(X)$ is an invariant of the action of $G$ on $X$. In the case when the action is finite, we compute the dimension of $L^G(X)$ in terms of fixed points of $X$ and prove several prominent results for $L^G(X)$, including Bessel's inequality and Frobenius reciprocity.
\end{abstract}
\\[3pt]

\section{Introduction}
Let $G$ be a finite group and let $H$ be a subgroup of $G$. Denote by $C(G)$ and $C(H)$ the spaces of complex-valued class functions on $G$ and on $H$, respectively. {\it Frobenius reciprocity} for class functions on $G$ states that $\Res{G}{H}{}$ and $\Ind{G}{H}{}$ are Hermitian adjoint with respect to the Hermitian inner product defined by 
\begin{equation}
\gen{f, g} = \dfrac{1}{\abs{G}}\lsum{x\in G}{}f(x)\lbar{g(x)}\quad\textrm{and}\quad \gen{h, k}_H = \dfrac{1}{\abs{H}}\lsum{x\in H}{}h(x)\lbar{k(x)}
\end{equation}
for all $f, g\in C(G), h, k\in C(H)$. In other words, if $f$ is a class function on $H$ and if $g$ is a class function on $G$, then
\begin{equation}
\gen{\Ind{G}{H}{f}, g}_H = \gen{f, \Res{G}{H}{g}},
\end{equation}
where $\Res{G}{H}{}$ is a linear transformation from $C(G)$ to $C(H)$ and $\Ind{G}{H}{}$ is a linear transformation from $C(H)$ to $C(G)$. This result is crucial and plays fundamental roles in \mbox{proving} well-known results in the representation theory of finite groups such as Mackey's irreducibility criterion; see, for instance, \cite[Theorem 8.3.6]{BS2012RTF}.

It is well known that the conjugation relation in any group $G$ may be viewed as a group action of $G$ on itself by the formula $g\cdot x = gxg^{-1}$ for all $g, x\in G$. This suggests studying Frobenius reciprocity in the setting of group actions. In the present article, we generalize Frobenius reciprocity to the family of functions that are invariant under a given group action. We remark that there are other versions of Frobenius reciprocity; see, for instance, \cite{MR3977723, MR3592484}.

Let $\F$ be a field and let $X$ be a nonempty set with an action of a group $G$ (that is, $X$ is a $G$-set). Define
\begin{equation}
    L(X)=\cset{f}{f \textrm{ is a function from $X$ to $\F$}}.
\end{equation}
Recall that $L(X)$ is a vector space under the following vector addition and scalar multiplication:
\begin{align*}
    (f+g)(x) &=f(x)+g(x)\\
    (\alpha f)(x) &=\alpha f(x)
\end{align*}
for all $f,g\in L(X), \alpha\in\F, x\in X$. Furthermore, $G$ acts linearly on $L(X)$ by the formula
\begin{equation}\label{equ: linear action on G-set}
    (a\star f)(x) = f(a^{-1}\cdot x),\qquad x\in X,
\end{equation}
for all $a\in G, f\in L(X)$, where $\star$ is the induced $G$-action on $L(X)$ and $\cdot$ is the given $G$-action on $X$. Therefore, we can speak of the {\it fixed subspace} of $L(X)$:
$$
    \Fix{(L(X))} = \cset{f\in L(X)}{a\cdot f = f \textrm{ for all }a\in G}.
$$
It is not difficult to check that $a\cdot f = f$ for all $a\in G$ if and only if $f(a\cdot x)=f(x)$ for all $a\in G, x\in X$. The fixed subspace of $L(X)$ is so important that we give a separate definition.

\begin{definition}\label{def: L^G(X)}
Let $X$ be a $G$-set. The fixed subspace of $L(X)$ associated with the action given by \eqref{equ: linear action on G-set} is defined as
\begin{equation}
\LG{G}{X}=\cset{f\in L(X)}{f(a\cdot x)=f(x) \textrm{ for all }a\in G, x\in X}.
\end{equation}
\end{definition}

In the case when $X$ is a finite-dimensional vector space, \eqref{equ: linear action on G-set} induces an action of $G$ on the space $\F[X]$ of polynomial functions on $X$. The study of this action along with the corresponding fixed subspace is a fundamental topic in invariant theory \cite{MR2004511, MR1304906, MR1869812, MR1716962}. The following examples indicate that several familiar families  of functions in the literature may be viewed as $\LG{G}{X}$ with appropriate group actions.

\begin{example}
Let $X$ be a nonempty set. If $G=\set{\id{X}}$ is the trivial subgroup of $\sym{X}$, then $G$ acts on $X$ by evaluation: $\sigma\cdot x = \sigma(x)$ for all $\sigma\in G, x\in X$ and in this case, $\LG{G}{X}=L(X)$. If $G=\sym{X}$, then $\LG{G}{X}$ becomes the space of constant functions: $$\LG{G}{X}=\cset{f_{\alpha}}{\alpha\in\F},$$ where $f_{\alpha}(x)=\alpha$ for all $x\in X$. 
\end{example}

\begin{example}[Class functions]
Let $G$ be a group and let $\F = \C$. Recall that $G$ acts on itself by conjugation: $a\cdot x=axa^{-1}$ for all $a, x\in G$.
In this case, $$\LG{G}{G}=\cset{f\colon G\to \C}{f(axa^{-1})=f(x) \textrm{ for all }a, x\in G},$$ which is the family of complex-valued class functions defined on $G$.
\end{example}

\begin{example}[Periodic functions]
Let $\F$ be a field. Suppose that $A$ is an abelian group. Fix $t\in A$ and set $G=\gen{t}=\set{nt:n\in\Z}$. Then $G$ acts on $A$ by addition and
$$\LG{G}{A}=\cset{f:A\to\F}{f(x+t)=f(x) \textrm{ for all }x\in A},$$
which is the family of periodic functions defined on $A$ with period $t$.
\end{example}

\begin{example}[Modular functions]
Let $\C^{\infty}=\C\cup\set{\infty}$ be the extended complex plane. Recall that a modular function $f\colon \C^{\infty}\to\C^{\infty}$ must satisfy the condition that $$f\left(\dfrac{az+b}{cz+d}\right)=f(z),\qquad z\in\C^{\infty},$$
where $a,b,c,d\in \Z$ and $ad-bc=1$ \cite[p. 34]{MR1027834}. Hence, modular functions are elements in $L^G(\C^{\infty})$, where $G$ is the modular group consisting of matrices of the form
$$\begin{bmatrix}a & b \\ c & d\end{bmatrix}\quad\textrm{with } a,b,c,d\in\Z \textrm{ and } ad - bc = 1$$
and $G$ acts on $\C^{\infty}$ by the formula
$$
\begin{bmatrix}a & b \\ c & d\end{bmatrix}\cdot z = \dfrac{az+b}{cz+d}.
$$
\end{example}

\begin{example}
A gyrogroup is a nonassociative group-like structure that is not, in \mbox{general}, a group \cite{AU2008AHG}.  Nevertheless, it generalizes the notion of a group and shares several properties with groups. The present article stems from the study of left regular representation of a finite gyrogroup in a series of articles \cite{TS2018MTGP, TS2020CRG, TSLRG2020}. More precisely, let $(K, \oplus)$ be a gyrogroup. As in Section 4 of \cite{TS2018MTGP}, the space 
\begin{equation}
    \gyrL{K} = \cset{f\in L(K)}{f(a\oplus\gyr{x, y}{z}) = f (a\oplus z)\textrm{ for all }a, x, y, z \in K}
\end{equation}
arises as a representation space of $K$ associated with a gyrogroup version of left regular representation. Using the change of variable, $w = a\oplus z$, we obtain that
$$
\gyrL{K} = \cset{f\in L(K)}{f(a\oplus\gyr{x, y}{(\ominus a\oplus z)}) = f (z) \textrm{ for all }a, x, y, z \in K}.
$$
Let $G$ be the subgroup of $\sym{K}$ generated by the set $\cset{L_a\circ\gyr{x, y}{}\circ L_{a}^{-1}}{a, x, y\in K}$, where $L_a$ is the left gyrotranslation by $a$ defined by $L_a(z) = a \oplus z$ for all $z\in K$ and $\gyr{x, y}{}$ is the gyroautomorphism generated by $x$ and $y$. It is clear that $G$ acts on $K$ by evaluation. Furthermore, $\gyrL{K} = \LG{G}{K}$.
\end{example}

In Section \ref{sec: basic property}, we study basic properties of arbitrary group actions related to their corresponding fixed subspaces. In Section \ref{sec: finite action}, we reduce to the case of a finite action (that is, an action of a finite group on a finite set) and compute the dimension of the fixed subspace. This leads to some remarkable properties of fixed-point-free actions. Once the usual Hermitian inner product on $L(X)$ is introduced, where $X$ is a finite $G$-set, an orthogonal decomposition of $L(X)$ is obtained and several interesting results such as Fourier expansion, Bessel's inequality, and Frobenius reciprocity are established for the case of functions invariant under the action of $G$ on $X$.

\section{Basic properties of group actions and their corresponding spaces}\label{sec: basic property}

Let $G$ be a group and let $X$ be a $G$-set. Recall that the action of $G$ on $X$ induces an equivalence relation $\sim$ given by
\begin{equation}\label{def: sim relation}
    x\sim y \quad\textrm{if and only if}\quad y=a\cdot x \textrm{ for some } a\in G
\end{equation}
for all $x,y\in X$. Also, the orbit of $x\in X$ under the action of $G$ is defined as
$$\orb{x}=\set{y\in X:y\sim x}=\set{a\cdot x:a\in G}.$$
Hence, the collection $\cset{\orb{x}}{x\in X}$ forms a partition of $X$. This partition leads to a characterization of elements in $\LG{G}{X}$, as shown in the next theorem, and eventually to a standard basis for $\LG{G}{X}$ if there are only finitely many orbits of $G$ on $X$.

\begin{theorem}\label{thm: f in L^G(x) iff f is constant}
Suppose that $X$ is a $G$-set and let $\cols{P}$ be the partition of $X$ determined by the equivalence relation \eqref{def: sim relation}. If $f\in L(X)$, then $f\in \LG{G}{X}$ if and only if $f$ is constant on $C$ for all $C\in \cols{P}$.
\end{theorem}
\begin{proof}
Suppose that $f\in L^G(X)$ and let $C\in \cols{P}$. Let $x,y\in C$. Then $y=a\cdot x$ for some $a\in G$. Thus, $f(y)=f(a\cdot x)=f(x)$. This proves that $f$ is constant on $C$.

Suppose conversely that $f$ is constant on $C$ for all $C\in \cols{P}$. Let $x\in X$ and let $a\in G$. Since $x\sim a\cdot x$, it follows that $x$ and $a\cdot x$ belong to the same orbit $C$ in $\cols{P}$. By assumption, $f$ is constant on $C$ and so $f(x)=f(a\cdot x)$. Since $x$ and $a$ are arbitrary, we obtain that $f\in L^G(X)$.
\end{proof}

In view of Theorem \ref{thm: f in L^G(x) iff f is constant}, a natural question arises: Can a space of functions on a set endowed with a partition be viewed as $\LG{G}{X}$ for a suitable group action? The answer to this question is affirmative. In fact, by Corollary 1.17  of \cite{MR1716962}, if $X$ is a nonempty set and if $\cols{P} = \cset{X_i}{i\in I}$ is a partition of $X$, then the following permutation group 
\begin{equation}
S_{\cols{P}} = \cset{\sigma\in\sym{X}}{\sigma(X_i) = X_i \textrm{ for all }i\in I}    
\end{equation}
acts on $X$ by evaluation and induces its orbits on $X$ as the cells of the partition. The following theorem shows that the group $S_{\cols{P}}$ may be replaced by its subgroup $G_{\cols{P}}$ and, among other things, describes a concrete method to construct the group $G_{\cols{P}}$. We remark that if $X$ is finite, then $G_{\cols{P}}$ and $S_\cols{P}$ coincide.

\begin{theorem}\label{thm: group induced by partition of set}
Let $X$ be a nonempty set and let $\cols{P} = \cset{[x]}{x\in X}$ be a partition of $X$. Then there exists a subgroup $G_{\cols{P}}$ of $S_{\cols{P}}$ that acts on $X$ such that $\orb{x} = [x]$ for all $x\in X$.
\end{theorem}
\begin{proof}
Define an equivalence relation $\sim_{\cols{P}}$ on $X$ by $y\sim_{\cols{P}} z$ if and only if $y\in[x]$ and $z\in[x]$ for some $x\in X$. For all $x,y\in X$, define 
\begin{equation*}
\tau(x,y)=
\begin{cases}
\id{X} & \textrm{if } x = y \textrm{ or } x\nsim_{\cols{P}} y;\\
(x~~y) & \textrm{otherwise}.
\end{cases}
\end{equation*}
Then $\tau(x,y)\in \sym{X}$. Set $G_{\cols{P}}=\gen{\tau(x,y):x,y\in X}$. Then $G_{\cols{P}}$ acts on $X$ by evaluation. Next, we prove that $\orb{x}=[x]$ for all $x\in X$. Let $x\in X$. By definition, $[x]\subseteq\orb{x}$. To prove the reverse inclusion, suppose that $y\in\orb{x}$. Then $y=\tau\cdot x$ for some $\tau\in G$. We claim that $y\sim_{\cols{P}} x$. If $x=y$, then we are done. We may therefore assume that $x\neq y$. By construction, $\tau=\tau(y_1,y_2)\circ\tau(y_3,y_4)\circ\cdots\circ\tau(y_{2m-1},y_{2m})$. Moreover, we can assume that $\tau(y_{2i-1},y_{2i})\neq \id{X}$ for all $i=1,2,\ldots,m$. Therefore,
$$
\tau = (y_1~~ y_2)(y_3~~ y_4)\cdots (y_{2m-1}~~ y_{2m}).
$$
Since $y=\tau(x)$, there is a maximum value $j_1\in\set{1,2,\ldots, 2m}$ such that $y_{j_1}\neq x$ and $(y_{j_1}~~x)$ is a transposition factor in $\tau$. So $y_{j_1}\sim_{\cols{P}} x$. If $y_{j_1}=y$, then $y\sim_{\cols{P}} x$. If $y_{j_1}\neq y$, then there is a maximum value $j_2\in\set{1, 2,\ldots, 2m}$ such that $1\leq j_2 < j_1$ and $(y_{j_2}~~y_{j_1})$ is a transposition factor in $\tau$. So $y_{j_2}\sim_{\cols{P}} y_{j_1}\sim_{\cols{P}} x$. Continuing this procedure, we obtain that $y\sim_{\cols{P}} y_{j_k}\sim_{\cols{P}} y_{j_{k-1}}\sim_{\cols{P}}\cdots\sim_{\cols{P}} y_{j_2}\sim_{\cols{P}} y_{j_1}\sim_{\cols{P}} x$. Hence, $y\in [x]$. This proves $\orb{x}\subseteq [x]$ and so equality holds.

Next, we prove that $\tau(x,y)\in S_{\cols{P}}$ for all $x,y\in X$. Hence, $G_{\cols{P}}\subseteq S_{\cols{P}}$  by the minimality of $G_{\cols{P}}$. Let $x, y\in X$. If $x=y$ or $x\nsim_{\cols{P}} y$, then $\tau(x,y)=\id{X}$ and hence $\tau(x, y)(X_i) = X_i$ for all $X_i\in \cols{P}$. We may therefore assume that $x\neq y$ and $x\sim_{\cols{P}} y$. Let $X_i\in \cols{P}$ and let $z\in X_i$. Note that
\begin{equation*}
\tau(x,y)(z)=
\begin{cases}
z & \textrm{if } z\neq x \textrm{ and } z\neq y;\\
y & \textrm{if } z=x;\\
x & \textrm{if } z=y.
\end{cases}
\end{equation*}
Since $x\sim_{\cols{P}} y$, $x$ and $y$ are in the same cell of $\cols{P}$. It follows that $\tau(x,y)(z)\in X_i$. This implies that $\tau(x, y)(X_i) = X_i$, which completes the proof.
\end{proof}

Let $X$ be a $G$-set and let $\cols{P}$ be the partition of $X$ determined by the equivalence relation \eqref{def: sim relation}. For each $C\in \cols{P}$, the indicator function $\delta_C$ is defined by
\begin{equation}
\delta_C(x) = 
\begin{cases}
1 & \textrm{if } x\in C;\\
0 & \textrm{if } x\in X\setminus C.
\end{cases}
\end{equation}
By Theorem \ref{thm: f in L^G(x) iff f is constant}, $\delta_C$ belongs to $\LG{G}{X}$ for all $C\in \cols{P}$. In fact, we obtain the following theorem.

\begin{theorem}\label{thm: basis of L^G(X)}
Suppose that $X$ is a $G$-set and let $\cols{P}$ be the partition of $X$ determined by the equivalence relation \eqref{def: sim relation}. Then $\cols{B}=\cset{\delta_C}{C\in \cols{P}}$ is a linearly independent set in $\LG{G}{X}$. Furthermore, $\cols{P}$ is finite if and only if $\cols{B}$ forms a basis for $\LG{G}{X}$. In particular,  $\dim(\LG{G}{X})\geq\abs{\cols{P}}$ and equality holds if $\cols{P}$ is finite.
\end{theorem}
\begin{proof}
By definition, $\cols{B}$ is linearly independent. Assume that $\cols{P}$ is finite, say $\cols{P}=\set{C_1,C_2,\ldots,C_n}$. Fix $c_i\in C_i$ for all $i=1,2,\ldots,n$. Then $$f=f(c_1)\delta_{C_1}+f(c_2)\delta_{C_2}+\cdots +f(c_n)\delta_{C_n}$$
for all $f\in\LG{G}{X}$ and so $\cols{B}$ spans $\LG{G}{X}$. To prove the converse, suppose that $\cols{P}$ is infinite. Assume to the contrary that $\cols{B}$ forms a basis for $\LG{G}{X}$. Define $f$ by $f(x)=1$ for all $x\in X$. Then $f\in\LG{G}{X}$ and so $f=a_1\delta_{C_1}+a_2\delta_{C_2}+\cdots+a_n\delta_{C_n}$ for some $C_1,C_2,\ldots,C_n\in\cols{P}$. Since $\cols{P}$ is infinite, there is an orbit $C\in\cols{P}\backslash\set{C_1,C_2,\ldots,C_n}$. Choose $c\in C$. Then $f(c)=1$, whereas
$$
(a_1\delta_{C_1}+a_2\delta_{C_2}+\cdots+a_n\delta_{C_n})(c) = a_1\delta_{C_1}(c)+a_2\delta_{C_2}(c)+\cdots+a_n\delta_{C_n}(c) = 0.
$$
Hence, $f\neq a_1\delta_{C_1}+a_2\delta_{C_2}+\cdots+a_n\delta_{C_n}$, a contradiction. This shows that $\cols{B}$ is not a basis for $\LG{G}{X}$.

Since $\cols{B}$ is linearly independent, it follows that $\dim{(L^G(X))}\geq\abs{\cols{B}} = \abs{\cols{P}}$. Moreover, if $\cols{P}$ is finite, then $\cols{B}$ is a basis for $L^{G}(X)$ and so $\dim{(L^G(X))} = \abs{\cols{P}}$.
\end{proof}

According to Theorem \ref{thm: basis of L^G(X)}, $\cset{\delta_C}{C\in \cols{P}}$ does not form a basis for $\LG{G}{X}$ in the case when $\cols{P}$ is infinite. It turns out that $\cset{\delta_C}{C\in \cols{P}}$ forms a basis for the following subspace of $\LG{G}{X}$:
\begin{equation}
    L^G_{\rm fs}(X)=\set{f\in\LG{G}{X}:f\textrm{ is nonzero on finitely many orbits in }X}
\end{equation}
so that the dimension of $L^G_{\rm fs}(X)$ equals $\abs{\cols{P}}$. Next, we mention some related properties between group actions and their corresponding spaces.

\begin{theorem}\label{thm: characterization of L(X) = LG(X)}
Let $G$ be a group and let $X$ be a $G$-set. Then the following are equi-valent:
\begin{enumerate}
    \item\label{item: equality of L(X) and LG(X)} $L(X)=\LG{G}{X}$;
    \item\label{item: orbit singleton} $\abs{\orb{x}}=1$ for all $x\in X$;
    \item\label{item: act trivially} $G$ acts trivially on $X$.
\end{enumerate}
\end{theorem}
\begin{proof}
To prove the equivalence \eqref{item: equality of L(X) and LG(X)} $\Leftrightarrow$ \eqref{item: orbit singleton}, suppose that $L(X)=\LG{G}{X}$. Let $x\in X$. Define $\delta_x$ by 
$$
\delta_x(z)= 
\begin{cases}
1 & \textrm{if }z = x;\\
0 & \textrm{otherwise}
\end{cases}
$$
for all $z\in X$. By assumption, $\delta_x\in\LG{G}{X}$. For each $y\in\orb{x}$, $\delta_x(y)=\delta_x(x)=1$ by Theorem \ref{thm: f in L^G(x) iff f is constant}. This implies $y=x$. Thus,  $\orb{x}=\set{x}$ and so $\abs{\orb{x}}=1$. Conversely, suppose that $\abs{\orb{x}}=1$ for all $x\in X$. Let $f\in L(X)$. Then $f(y)=f(z)$ for all $y,z\in\orb{x}$. By Theorem \ref{thm: f in L^G(x) iff f is constant}, $f\in\LG{G}{X}$. This proves $L(X)\subseteq\LG{G}{X}$ and so equality holds.

To prove the equivalence \eqref{item: orbit singleton} $\Leftrightarrow$ \eqref{item: act trivially}, suppose that $\abs{\orb{x}}=1$ for all $x\in X$. Since $\set{x} = \orb{x} = \cset{a\cdot x}{a\in G}$, we obtain that $a\cdot x=x$ for all $a\in G$. Hence, $G$ acts trivially on $X$. Conversely, suppose that $G$ acts trivially on $X$. Let $x\in X$. Then $a\cdot x=x$ for all $a\in G$. This implies $\orb{x}=\set{x}$ and so $\abs{\orb{x}}=1$.
\end{proof}

\begin{theorem}\label{thm: characterization of transitive action}
Let $\F$ be a field and let $X$ be a $G$-set. Then the following are equivalent:
\begin{enumerate}
    \item\label{action of G on X is transitive} The action of $G$ on $X$ is transitive;
    \item\label{dim(L^G(X)=1)} $\dim{(\LG{G}{X})}=1$;
    \item\label{L^G(X)=gen{f_1}} $\LG{G}{X}=\set{f_{\alpha}:\alpha\in \F}=\spn{f_1}$, where $f_{\alpha}(x)=\alpha$ for all $x\in X$, $\alpha\in\F$.
\end{enumerate}
\end{theorem}
\begin{proof}
Let $x,y\in X$. By assumption, there is an element $a\in G$ such that $y=a\cdot x$; that is, $x\sim y$. Thus, $x$ and $y$ are in the same orbit. This shows that $X$ has only one orbit. By Theorem \ref{thm: basis of L^G(X)}, $\dim{(\LG{G}{X})}=1$. This proves that \eqref{action of G on X is transitive} implies \eqref{dim(L^G(X)=1)}.

It is clear that $f_{\alpha}\in\LG{G}{X}$ for all $\alpha\in\F$. Thus, $\set{f_{\alpha}:\alpha\in\F}\subseteq\LG{G}{X}$. Let $f\in\LG{G}{X}$ and let $x,y\in X$. Since $\dim{(\LG{G}{X})}=1$, $X$ has only one orbit and so $y\sim x$. Hence, there is an element $a\in G$ such that $y=a\cdot x$. It follows that $f(y)=f(a\cdot x)=f(x)$ and so $f$ is constant. Therefore, $\LG{G}{X}\subseteq\set{f_{\alpha}:\alpha\in\F}$. This proves that \eqref{dim(L^G(X)=1)} implies  \eqref{L^G(X)=gen{f_1}}.

Let $x,y\in X$. Since $\LG{G}{X}=\spn{f_1}$, $\dim{\LG{G}{X}}=1$. By Theorem \ref{thm: basis of L^G(X)}, $X$ has only one orbit. Hence, $x\sim y$ and so there is an element $a\in G$ such that $y=a\cdot x$. Therefore, the action of $G$ on $X$ is transitive. This proves that  \eqref{L^G(X)=gen{f_1}} implies \eqref{action of G on X is transitive}.
\end{proof}

We close this section with the following result, which indicates that $\LG{G}{X}$ is an {\it invariant} of the action of $G$ on $X$. Therefore, in certain circumstances, one can use the notion of $\LG{G}{X}$ to distinguish inequivalent group actions.

\begin{proposition}\label{prop: equivalent G-set imply isomorphic vector space}
Let $X$ and $Y$ be $G$-sets. If $\Phi\colon X\to Y$ is an equivalence, then the map $\tau$ defined by
\begin{equation}
    \tau(f) = f\circ \Phi^{-1},\qquad f\in L(X),
\end{equation}
is a linear isomorphism from $L(X)$ to $L(Y)$ that restricts to a linear isomorphism from $\LG{H}{X}$ to $\LG{H}{Y}$ for any subgroup $H$ of $G$. Consequently, if $X$ and $Y$ are equivalent, then $L(X)\cong L(Y)$ and $\LG{H}{X}\cong \LG{H}{Y}$ as vector spaces for any subgroup $H$ of $G$.  
\end{proposition}
\begin{proof}
The proof that $\tau$ is a linear isomorphism is straightforward. Let $H$ be a subgroup of $G$ and let $f\in\LG{H}{X}$. We claim that $\tau(f)\in\LG{H}{Y}$. Let $a\in H$ and let $y\in Y$. By surjectivity, there is an element $x\in X$ such that $y=\Phi(x)$. Thus, $\tau(f)(a\cdot y)=\tau(f)(a\cdot\Phi(x))=\tau(f)(\Phi(a\cdot x))=f(\Phi^{-1}(\Phi(a\cdot x)))=f(a\cdot x)=f(x)=f(\Phi^{-1}(y))=(f\circ\Phi^{-1})(y)=\tau(f)(y)$. Hence, $\tau(f)\in\LG{H}{Y}$ and so $\tau$ maps $\LG{H}{X}$ to $\LG{H}{Y}$.

Let $g\in \LG{H}{Y}$ and set $f=g\circ\Phi$. Note that $f$ is a map from  $X$ to $\F$ and that $f(a\cdot x)=g(\Phi(a\cdot x))=g(a\cdot\Phi(x))=g(\Phi(x))=f(x)$ for all $a\in H$ and $x\in X$. Hence, $f\in \LG{H}{X}$. Furthermore, $\tau(f)=f\circ\Phi^{-1}=(g\circ\Phi)\circ\Phi^{-1}=g$. This proves that $\tau$ is surjective. Therefore, the restriction $\tau:\LG{H}{X}\to\LG{H}{Y}$ is a linear isomorphism.
\end{proof}

\newpage

The converse to Proposition \ref{prop: equivalent G-set imply isomorphic vector space} is not, in general, true. That is, the condition that \qt{$\LG{H}{X}\cong \LG{H}{Y}$ as vector spaces for some subgroup $H$ of $G$} does not imply that \qt{$X\cong Y$ as $G$-sets}. In fact, let $X$ be a set having at least two distinct elements, namely that $x, y\in X$ and $x\ne y$. Then $\sym{X}$ acts transitively on $\set{x}$ and on $\set{x, y}$ by evaluation. By Theorem \ref{thm: characterization of transitive action}, $\dim{(\LG{\sym{X}}{\set{x}})} = 1 = \dim{(\LG{\sym{X}}{\set{x, y}})}$ and so $\LG{\sym{X}}{\set{x}}\cong \LG{\sym{X}}{\set{x, y}}$. However, $\set{x}$ and $\set{x, y}$ are not equivalent.

\section{The case of finite actions}\label{sec: finite action}
If $G$ is a finite group and if $X$ is a finite $G$-set (that is, if the action is finite), we may use the Cauchy--Frobenius lemma (also called the Burnside lemma) to compute the dimension of $\LG{G}{X}$. Moreover, the space $L(X)$ (and hence also $\LG{G}{X}$) possesses a standard Hermitian inner product (the base field is assumed to be the field of complex numbers). This allows us to prove further related properties between group actions and their corresponding spaces, including Bessel's inequality and Frobenius reciprocity.

\subsection{Dimensions and fixed points}
With Theorem \ref{thm: basis of L^G(X)} in hand, we give a formula for computing the dimension of $\LG{G}{X}$, where $G$ and $X$ are finite, in terms of fixed points of $X$. As a consequence of this result, we obtain a few remarkable properties of free (also called fixed-point-free) actions.

\begin{lemma}\label{lem: dimension in terms of number of fixed point}
Let $G$ be a finite group and let $X$ be a finite $G$-set. For any subgroup $H$ of $G$,
\begin{equation}
    \dim{(\LG{H}{X})} = \dfrac{1}{\abs{H}}\lsum{a\in H}{}\abs{\Fix{a}},
\end{equation}
where $\Fix{a} = \cset{x\in X}{a\cdot x = x}$.
\end{lemma}
\begin{proof}
Recall that $H$ acts on $X$ by the action inherited from $G$ since $H$ is a subgroup of $G$. Let $\aorb{H}{x}=\set{a\cdot x:a\in H}$ and let $\cols{P}=\set{\aorb{H}{x}:x\in X}$. By Theorem \ref{thm: basis of L^G(X)}, $\dim(\LG{H}{X})$ equals $\abs{\cols{P}}$, the number of orbits of $H$ on $X$. By the famous Cauchy--Frobenius lemma, $\abs{\cols{P}}=\dfrac{1}{\abs{H}}\lsum{a\in H}{}\abs{\Fix{a}}$.
\end{proof}

\begin{lemma}\label{lem: diferrence of dimension}
Let $G$ be a finite group and let $X$ be a finite $G$-set. For any subgroup $H$ of $G$,
\begin{equation}
\abs{G}\dim{(\LG{G}{X})} - \abs{H}\dim{(\LG{H}{X})} = \lsum{a\in G\setminus H}{}\abs{\Fix{a}}.
\end{equation}
\end{lemma}
\begin{proof}
Note that $\mathrm{Fix}_H\,a= \mathrm{Fix}_G\,a$ for all $a\in H$ because $H$ acts on $X$ by the action inherited from $G$. It follows from Lemma \ref{lem: dimension in terms of number of fixed point} that
\begin{equation*}\tag*{$\square$}
   \abs{G}\dim{(\LG{G}{X})} - \abs{H}\dim{(\LG{H}{X})} = \lsum{a\in G}{}\abs{\Fix{a}}-\lsum{a\in H}{}\abs{\Fix{a}}= \lsum{a\in G\setminus H}{}\abs{\Fix{a}}. 
\end{equation*}
\let\qed\relax
\end{proof}

Recall that an action of a group $G$ on a set $X$ is {\it free} if $\stab{x} = \set{e}$ for all $x\in X$; that is, if for all $a\in G, x\in X$, $a\cdot x = x$ implies $a = e$. It is clear that an action of $G$ on $X$ is free if and only if $\Fix{a} = \emptyset$ for all $a\in G\setminus\set{e}$. By Lemma \ref{lem: diferrence of dimension}, the ratio of $\dim{(\LG{H}{X})}$ and $\dim{(\LG{G}{X})}$ is simply the index of $H$ in $G$ when the action of $G$ on $X$ is free. It turns out that the family of free actions is rather limited in the sense of Corollary \ref{cor: the number of orbits in X with respect to G}.

\begin{theorem}\label{thm:fraction of dimension}
Let $G$ be a finite group with a subgroup $H$ and let $X$ be a finite nonempty set. If $G$ acts freely on $X$, then
\begin{equation}
    \dfrac{\dim{(\LG{H}{X})}}{\dim{(\LG{G}{X})}} = [G\colon H],
\end{equation}
where $[G\colon H]$ denotes the index of $H$ in $G$.
\end{theorem}
\begin{proof}
Since $G$ acts freely on $X$, $\Fix{a}=\emptyset$ for all $a\in G\setminus H$. By Lemma \ref{lem: diferrence of dimension}, $$\abs{G}\dim{(\LG{G}{X})} - \abs{H}\dim{(\LG{H}{X})} = \lsum{a\in G\setminus H}{}\abs{\Fix{a}}=0.$$ Hence, $\dfrac{\dim{(\LG{H}{X})}}{\dim{(\LG{G}{X})}} = \dfrac{\abs{G}}{\abs{H}} = [G\colon H]$.
\end{proof}

\begin{corollary}\label{cor: the number of orbits in X with respect to G}
If a finite group $G$ acts freely on a finite nonempty set $X$, then 
\begin{equation}
\textrm{the number of orbits of $G$ on $X$} = \dim{(\LG{G}{X})} = \dfrac{\abs{X}}{\abs{G}}. 
\end{equation}
Therefore, if $\abs{G}$ does not divide $\abs{X}$, then $G$ does not act freely on $X$.
\end{corollary}
\begin{proof}
Note that $\LG{\set{e}}{X} = L(X)$. Since $X$ is finite, it follows that $\dim{(L(X))} = \abs{X}$. By Theorem \ref{thm:fraction of dimension}, $\dfrac{\abs{X}}{\dim{(\LG{G}{X})}}=\dfrac{\dim{(\LG{\set{e}}{X})}}{\dim{(\LG{G}{X})}} = [G\colon \set{e}]=\abs{G}$ and the corollary follows.
\end{proof}

Corollary \ref{cor: the number of orbits in X with respect to G} provides a numerical condition that can be used to verify that a given action is not free. Hence, in this case, a nontrivial stabilizer subgroup exists and some fixed-point set is nonempty. For instance, by Corollary \ref{cor: the number of orbits in X with respect to G}, we know that the following actions are not free without having to find explicit stabilizer subgroups (or fixed-point sets).
\begin{enumerate}
    \item The action of $\sym{X}$ on $X$ in the case when $X$ is finite and $\abs{X}>2$.
    \item Let $G$ be a finite group and let $H$ be a nontrivial subgroup of $G$. Then $G$ acts on the set of left cosets, $G/H = \cset{xH}{x\in G}$, by left multiplication: $a\cdot (xH) = (ax)H$ for all $a, x\in G$.
    \item Let $G$ be a finite group and let $Y$ be a subset of $G$ with nontrivial normalizer (that is, the normalizer subgroup $N_G(Y)$ does not equal $\set{e}$). For example, we may let $Y$ be a nontrivial subgroup of $G$. Then $G$ acts on the set $X = \cset{xYx^{-1}}{x\in G}$ by conjugation: $a\cdot (xYx^{-1}) = (ax)Y(ax)^{-1}$ for all $a, x\in G$.
    \item Let $G$ be a finite group whose order is divisible by a prime $p$. Then $G$ acts on the set of Sylow $p$-subgroups of $G$ by conjugation. 
    \item Let $G$ be a finite group whose order is divisible by a prime $p$. Then $G$ acts on the set $X= \cset{x\in G}{\abs{x} = p}$ by conjugation. 
    \item Let $\F$ be a finite field and let $G$ be the general linear group $\Gl{n, \F}$ with $n\geq 2$. Then $G$ acts on the set
    $$
    X = \Cset{\begin{bmatrix}\alpha_1\\ \alpha_2\\ \vdots \\ \alpha_n\end{bmatrix}}{\alpha_i\in\F\textrm{ for all }i = 1,2,\ldots, n}
    $$
    by (left) matrix multiplication: $A\cdot [\alpha_1~~\alpha_2~~\cdots~~\alpha_n]^t = A[\alpha_1~~\alpha_2~~\cdots~~\alpha_n]^t$ for all $A\in \Gl{n, \F}$, $[\alpha_1~~\alpha_2~~\cdots~~\alpha_n]^t\in X$.
   \item Let $G$ be a finite group that is not a $2$-group and let $X$ be a $G$-set. Then $G$ acts on the collection of subsets of $X$ by letting $a\cdot Y = \cset{a\cdot y}{y\in Y}$ for all $a\in G$ and for all subsets $Y$ of $X$. 
    \item Let $G$ be a finite group. Then $G\times G$ acts on $G$ by
$(a, b)\cdot x = axb^{-1}$ for all $(a, b)\in G\times G, x\in G$ (cf. Section 1.2.8 of \cite{MR1716962}).
\end{enumerate}

We close this section with another application of Lemma \ref{lem: dimension in terms of number of fixed point}.
\begin{proposition}
Let $G$ be a finite group and let $A$ be a group of automorphisms of $G$, that is, $A\leq \aut{G}$. Then
\begin{equation}
    \lsum{\tau\in A}{}\abs{\cset{x\in G}{\tau(x) = x}} \leq \abs{A}\abs{G}.
\end{equation}
\end{proposition}
\begin{proof}
Recall that $\aut{G}$ acts on $G$ by $\tau\cdot a=\tau(a)$ for all $\tau\in \aut{G}, a\in G$. Since $A$ is a subgroup of $\aut{G}$, by Lemma \ref{lem: dimension in terms of number of fixed point}, $\dim{(\LG{A}{G})} = \dfrac{1}{\abs{A}}\lsum{\tau\in A}{}\abs{\cset{x\in G}{\tau(x) = x}}$. Since $\LG{A}{G}\subseteq L(G)$, it follows that $\dim(\LG{A}{G}) \leq \dim{(L(G))} =\abs{G}$. Hence, the inequality follows.
\end{proof}

\subsection{Orthogonal decomposition and Frobenius reciprocity}
In this section, let $G$ be a (finite or infinite) group and let $X$ be a finite $G$-set unless otherwise stated. We also suppose that $\F = \C$. Thus,
\begin{equation}
   L(X)=\cset{f}{f \textrm{ is a function from $X$ to $\C$}}
\end{equation}
and $L(X)$ admits the complex inner product structure. In fact, for all $f, g\in L(X)$, define
\begin{equation}\label{eqn: inner product on L(X)}
    \gen{f, g} = \dfrac{1}{\abs{X}}\lsum{x\in X}{}f(x)\lbar{g(x)},
\end{equation}
where $\bar{\cdot}$ denotes complex conjugation. Then the following result is obtained.

\begin{proposition}\label{prop: inner product on LG(X) and its orthonormal basis}
Equation \eqref{eqn: inner product on L(X)} defines a Hermitian inner product on $L(X)$. Hence, $L(X)$ forms a complex inner product space. If $\cols{P} = \cset{\orb{x}}{x\in X}$, then $$\cols{B} = \Cset{\sqrt{\dfrac{\abs{X}}{\abs{C}}}\delta_C}{C\in \cols{P}}$$ forms an orthonormal basis for $\LG{G}{X}$.
\end{proposition}
\begin{proof}
The proof that \eqref{eqn: inner product on L(X)} defines a Hermitian inner product on $L(X)$ is direct computation. By Theorem \ref{thm: basis of L^G(X)}, $\cols{B}$ forms a basis for $\LG{G}{X}$. Next, we prove that $\cols{B}$ is orthonormal. Let $\sqrt{\dfrac{\abs{X}}{\abs{C}}}\delta_{C},\sqrt{\dfrac{\abs{X}}{\abs{D}}}\delta_{D}\in \cols{B}$ with $C, D\in\cols{P}$. If $C\neq D$, then
$$
 \gen{\sqrt{\dfrac{\abs{X}}{\abs{C}}}\delta_{C},\sqrt{\dfrac{\abs{X}}{\abs{D}}}\delta_{D}} = \dfrac{1}{\sqrt{\abs{C}\abs{D}}}\lsum{x\in X}{}\delta_{C}(x)\lbar{\delta_{D}(x)} = 0
$$
because for each $x\in X$, either $x\in C$ or $x\in D$. If $C=D$, then
$$
\gen{\sqrt{\dfrac{\abs{X}}{\abs{C}}}\delta_{C},\sqrt{\dfrac{\abs{X}}{\abs{D}}}\delta_{D}} = \dfrac{1}{\sqrt{\abs{C}\abs{D}}}\lsum{x\in X}{}\delta_{C}(x)\lbar{\delta_{D}(x)} = \dfrac{1}{\abs{C}}\lsum{x\in X}{}\abs{\delta_{C}(x)}^2 = 1
$$
because $\lsum{x\in X}{}\abs{\delta_{C}(x)}^2 = \abs{C}$. 
\end{proof}

One advantage of the inner product defined by \eqref{eqn: inner product on L(X)} is shown in the following theorem, demonstrating that the action of $G$ on $X$ is preserved by this inner product.

\begin{proposition}
The action given by \eqref{equ: linear action on G-set} is unitary in the sense that
$$\gen{a\cdot f, a\cdot g} = \gen{f, g}$$
for all $f, g\in L(X)$, $a\in G$. In particular, the map $f\mapsto a\cdot f$, $f\in L(X)$, is a unitary operator on $L(X)$ for all $a\in G$.
\end{proposition}
\begin{proof}
Let $f,g\in L(X)$ and let $a\in G$. Then
\begin{align*}
    \gen{a\cdot f, a\cdot g} &= \dfrac{1}{\abs{X}}\lsum{x\in X}{}(a\cdot f)(x)\lbar{(a\cdot g)(x)}\\
    {} &= \dfrac{1}{\abs{X}}\lsum{x\in X}{}f(a^{-1}\cdot x)\lbar{g(a^{-1}\cdot x)}\\
    {} &= \dfrac{1}{\abs{X}}\lsum{x\in X}{}f(x)\lbar{g(x)}\\
    {} &= \gen{f, g}.
\end{align*}
The third equality holds since the map $x\mapsto a^{-1}\cdot x$ is a bijection from $X$ to itself.
\end{proof}

To obtain an orthogonal decomposition of $L(X)$, we define a map $\sigma$ by
\begin{equation}\label{eqn: linear functional on L(X)}
    \sigma(f) = \lsum{x\in X}{}f(x),\qquad f\in L(X).
\end{equation}

\newpage

\begin{theorem}\label{thm: property of kernel sigma and Lgyr(G)}
Let $\sigma$ be the map defined by \eqref{eqn: linear functional on L(X)}. Then the following assertions hold:
\begin{enumerate}
\item\label{item: sigma linear functional}  $\sigma$ is a linear functional from $L(X)$ to $\C$.
\item\label{item: invariant subspace of sigma} $\ker{\sigma}$ is an invariant subspace of $L(X)$ under the action given by \eqref{equ: linear action on G-set}.
\item\label{item: ker sigma as perp} $\ker{\sigma} = (\spn{f_1})^\perp$, where $f_1(x) = 1$ for all $x\in X$.
\item\label{item: dimension of ker sigma} $\dim{(\ker{\sigma})} = \abs{X}-1$.
\item\label{item: invariant subspace of LG(X)} $\ker{\Bp{\res{\sigma}{\LG{G}{X}}}}$ is an invariant subspace of $\LG{G}{X}$ and its dimension equals the number of orbits on $X$ minus $1$. Here, $\res{\sigma}{\LG{G}{X}}$ is the restriction of $\sigma$ to $\LG{G}{X}$.
\item\label{item: LG(X) perp} $\LG{G}{X}^\perp \subseteq \ker{\sigma}$; equality holds if and only if the action of $G$ on $X$ is transitive.
\end{enumerate}
\end{theorem}
\begin{proof}
The proofs of Parts \eqref{item: sigma linear functional}, \eqref{item: ker sigma as perp}, and \eqref{item: invariant subspace of LG(X)} are immediate. Part \eqref{item: invariant subspace of sigma} holds because the map $x\mapsto a^{-1}\cdot x$ is a bijection from $X$ to itself. To prove Part \eqref{item: dimension of ker sigma}, note that $\spn{f_1}$ is a finite-dimensional subspace of $L(X)$. By the projection theorem in linear algebra and Part \eqref{item: ker sigma as perp}, $L(X) = \spn{f_1}\oplus \ker{\sigma}$ and so $\dim(\ker{\sigma})=\abs{X}-1$. That $\LG{G}{X}^\perp \subseteq \ker{\sigma}$ is clear. By Theorem \ref{thm: characterization of transitive action} and Part \eqref{item: ker sigma as perp}, $\LG{G}{X}^\perp = \ker{\sigma}$ if and only if $\LG{G}{X} = \spn f_1$ (since $\LG{G}{X}$ and $\spn{f_1}$ are finite-dimensional subspaces of $L(X)$) if and only if the action of $G$ on $X$ is transitive. This proves Part \eqref{item: LG(X) perp}.
\end{proof}

\begin{corollary}\label{cor: orthogonal direct sum of L(X) and LG(X)}
Let $G$ be a group and let $X$ be a finite $G$-set. Then
\begin{enumerate}
    \item\label{Projection of L(X) to LG(X)} $L(X) = \LG{G}{X}\operp \LG{G}{X}^\perp$;
    \item\label{Projection of L(X) to span f_1} $L(X) = \spn{f_1}\operp \ker{\sigma}$;
    \item\label{Projection of LG(X)} $\LG{G}{X} = \spn{f_1}\operp \ker{\Bp{\res{\sigma}{\LG{G}{X}}}}$.
\end{enumerate}
Here, $\operp$ denotes orthogonal direct sum decomposition.
\end{corollary}
\begin{proof}
Part \eqref{Projection of L(X) to LG(X)} follows from the projection theorem. Part \eqref{Projection of L(X) to span f_1} follows as in the proof of Theorem \ref{thm: property of kernel sigma and Lgyr(G)} \eqref{item: dimension of ker sigma}. Part \eqref{Projection of LG(X)} holds since $\spn{f_1}$ is a subspace of $\LG{G}{X}$.
\end{proof}

According to Proposition \ref{prop: inner product on LG(X) and its orthonormal basis}, we have an orthonormal basis for $\LG{G}{X}$, which is an orthonormal set in $L(X)$. Thus, several prominent results in linear algebra can be deduced from this fact.

\begin{theorem}
Let $G$ be a group and let $X$ be a finite $G$-set. Suppose that $$\cols{P} = \cset{\orb{x}}{x\in X} = \set{C_1, C_2,\ldots, C_n}.$$ Fix the ordered (orthonormal) basis of $\LG{G}{X}$: $\cols{B} = \left(\sqrt{\dfrac{\abs{X}}{\abs{C_1}}}\delta_{C_1}, \sqrt{\dfrac{\abs{X}}{\abs{C_2}}}\delta_{C_2},\ldots, \sqrt{\dfrac{\abs{X}}{\abs{C_n}}}\delta_{C_n}\right)$. Then the following assertions hold:
\begin{enumerate}
    \item\label{item: Fourier expansion} {\bf (Fourier expansion)} The Fourier expansion with respect to $\cols{B}$ of
a function \mbox{$f\in L(X)$} is
\begin{equation}
    \lhat{f} = \Bp{ \dfrac{1}{\abs{C_1}}\lsum{x\in C_1}{}f(x)}\delta_{C_1} + \Bp{ \dfrac{1}{\abs{C_2}}\lsum{x\in C_2}{}f(x)}\delta_{C_2} + \cdots + \Bp{ \dfrac{1}{\abs{C_n}}\lsum{x\in C_n}{}f(x)}\delta_{C_n};
\end{equation}
that is, the Fourier coefficients of $f$ are given by
\begin{equation}
    \gen{f, \sqrt{\dfrac{\abs{X}}{\abs{C_i}}}\delta_{C_i}} = \dfrac{1}{\sqrt{\abs{X}\abs{C_i}}}\lsum{x\in C_i}{}f(x)
\end{equation}
for all $i = 1,2,\ldots, n$.
\item\label{item: Bessel's inequality} {\bf (Bessel's inequality)} For all $f\in L(X)$,
\begin{equation}\label{ineq: Bessel's inequality}
    \lsum{i=1}{n}\dfrac{1}{\abs{C_i}}\Abs{\lsum{x\in C_i}{}f(x)}^2 \leq \lsum{x\in X}{}\abs{f(x)}^2.
\end{equation}
\item\label{item: Parseval identity not true} If $G$ acts nontrivially on $X$, then there exists a function $f\in L(X)$ with $\norm{\lhat{f}} < \norm{f}$. That is, the equality in Bessel's identity is not attained.
\end{enumerate}
\end{theorem}
\begin{proof} Recall that the Fourier coefficients of $f$ are $$\gen{f, \sqrt{\dfrac{\abs{X}}{\abs{C_i}}}\delta_{C_i}}  = \dfrac{1}{\abs{X}}\lsum{x\in C_i}{}f(x)\sqrt{\dfrac{\abs{X}}{\abs{C_i}}} = \dfrac{1}{\sqrt{\abs{X}\abs{C_i}}}\lsum{x\in C_i}{}f(x)$$ for all $i = 1,2,\ldots, n$. Hence, the Fourier expansion of $f$ with respect to $\cols{B}$ is
\begin{align*}
    \lhat{f} &= \gen{f, \sqrt{\dfrac{\abs{X}}{\abs{C_1}}}\delta_{C_1}}\sqrt{\dfrac{\abs{X}}{\abs{C_1}}}\delta_{C_1}+\cdots+\gen{f, \sqrt{\dfrac{\abs{X}}{\abs{C_n}}}\delta_{C_n}}\sqrt{\dfrac{\abs{X}}{\abs{C_n}}}\delta_{C_n}\\
    & = \Bp{\dfrac{1}{\sqrt{\abs{X}\abs{C_1}}}\lsum{x\in C_1}{}f(x)}\sqrt{\dfrac{\abs{X}}{\abs{C_1}}}\delta_{C_1}+\cdots+\Bp{ \dfrac{1}{\sqrt{\abs{X}\abs{C_n}}}\lsum{x\in C_n}{}f(x)}\sqrt{\dfrac{\abs{X}}{\abs{C_n}}}\delta_{C_n}\\
    &=\Bp{ \dfrac{1}{\abs{C_1}}\lsum{x\in C_1}{}f(x)}\delta_{C_1} + \cdots + \Bp{ \dfrac{1}{\abs{C_n}}\lsum{x\in C_n}{}f(x)}\delta_{C_n}.
\end{align*}
This proves Part \eqref{item: Fourier expansion}.

Recall that Bessel's inequality states that $\norm{\lhat{f}} \leq \norm{f}$. Hence,
\begin{align*}
    \lsum{i=1}{n}\Abs{\gen{f, \sqrt{\dfrac{\abs{X}}{\abs{C_i}}}\delta_{C_i}}}^2 \leq \gen{f,f}.
\end{align*}
Direct computation shows that $\lsum{i=1}{n}\Abs{\gen{f, \sqrt{\dfrac{\abs{X}}{\abs{C_i}}}\delta_{C_i}}}^2 = \dfrac{1}{\abs{X}}\lsum{i=1}{n}\dfrac{1}{\abs{C_i}}\Abs{\lsum{x\in C_i}{}f(x)}^2$ and that $\gen{f,f} = \dfrac{1}{\abs{X}}\lsum{x\in X}{}\abs{f(x)}^2$. Hence, \eqref{ineq: Bessel's inequality} follows. This proves Part \eqref{item: Bessel's inequality}.

Suppose that $G$ acts nontrivially on $X$. By Theorem \ref{thm: characterization of L(X) = LG(X)}, $\LG{G}{X} \subsetneq L(X)$. This implies that  $\cols{B}$ is not an orthonormal basis for $L(X)$ and so there exists a function $f$ in $L(X)$ with $\norm{\lhat{f}} < \norm{f}$ by Theorem 9.17 of \cite{SR2008ALA}. This proves Part \eqref{item: Parseval identity not true}.\end{proof}

Next, we extend Frobenius reciprocity from the space of class functions to that of functions invariant under a given group action in a natural way. Let $G$ be a group and let $X$ be a $G$-set. Recall that a (nonempty) subset $Y$ of $X$ is {\it invariant} if $a\cdot y\in Y$ for all $a\in G, y\in Y$; that is, if $G\cdot Y = Y$. It is not difficult to check that the following are equivalent:
\begin{enumerate}
\item\label{item: Y invariant subset} $Y$ is an invariant subset of $X$;
\item\label{item: characterization of invariant subset} for all $a\in G, x\in X$, $a\cdot x\in Y$ if and only if $x\in Y$.
\end{enumerate}

Let $X$ be a $G$-set and let $Y$ be an invariant subset of $X$. Define a map $\Res{X}{Y}{}$ from $L(X)$ to $L(Y)$ by
\begin{equation}
\Res{X}{Y}{f}(y) = f(y),\qquad y\in Y
\end{equation}
for all $f\in L(X)$. Also, for each $f\in L(Y)$, define $\tilde{f}$ by
\begin{equation}
\tilde{f}(x) = 
\begin{cases}
f(x) &\textrm{if }x\in Y;\\
0 &\textrm{otherwise}
\end{cases}
\end{equation}
for all $x\in X$. Then $\tilde{f}\in L(X)$. In fact, we have the following lemma.

\begin{lemma}\label{lem: epsilon is linear map from LG(Y) to LG(X)}
The map $\epsilon\colon L(Y)\to L(X)$ given by $\epsilon(f) = \tilde{f}$ is linear and maps $\LG{G}{Y}$ to $\LG{G}{X}$.
\end{lemma}
\begin{proof}
The proof that $\epsilon$ is linear is straightforward. By the remark above, $\epsilon(f)\in \LG{G}{X}$ for all $f\in\LG{G}{Y}$. 
\end{proof}

\begin{theorem}
Let $X$ be a $G$-set with an invariant subset $Y$. Then $\Res{X}{Y}{}\colon L(X)\to L(Y)$ is linear and maps $\LG{G}{X}$ surjectively onto $\LG{G}{Y}$.
\end{theorem}
\begin{proof}
The proof that $\Res{X}{Y}{}$ is linear is straightforward. Let $f\in \LG{G}{Y}$. By Lemma \ref{lem: epsilon is linear map from LG(Y) to LG(X)}, $\tilde{f}\in \LG{G}{X}$ and $\Res{X}{Y}{\tilde{f}(y)} = \tilde{f}(y)  = f(y)$ for all $y\in Y$. So $\Res{X}{Y}{\tilde{f}}=f$. This proves that $\Res{X}{Y}{}$ is surjective.
\end{proof}

Let $G$ be a finite group, let $X$ be a finite $G$-set, and let $Y$ be an invariant subset of $X$. Define a map $\Ind{X}{Y}{}$ on $L(Y)$ by
\begin{equation}\label{eqn: Induction map}
\Ind{X}{Y}{f}(x) = \dfrac{\abs{X}}{\abs{G}\abs{Y}}\lsum{b\in G}{}\tilde{f}(b^{-1}\cdot x),\qquad x\in X
\end{equation}
for all $f\in L(Y)$. Then $\Ind{X}{Y}{}$ is a linear transformation from $L(Y)$ to $\LG{G}{X}$, as shown in the following theorem. 

\begin{theorem}
The map $\Ind{X}{Y}{}$ defined by \eqref{eqn: Induction map} is a linear transformation from $L(Y)$ to $\LG{G}{X}$. 
\end{theorem}
\begin{proof}
The proof that $\Ind{X}{Y}{}$ is linear is straightforward. Let $f\in L(Y)$. Given $a\in G$ and $x\in X$, we have by inspection that 
\begin{align*}
    \Ind{X}{Y}{f(a\cdot x)} &= \dfrac{\abs{X}}{\abs{G}\abs{Y}}\lsum{b\in G}{}\tilde{f}(b^{-1}\cdot (a\cdot x))\\
    {} &= \dfrac{\abs{X}}{\abs{G}\abs{Y}}\lsum{b\in G}{}\tilde{f}((b^{-1}a)\cdot x)\\
    {} &= \dfrac{\abs{X}}{\abs{G}\abs{Y}}\lsum{c\in G}{}\tilde{f}(c^{-1}\cdot x)\\
    {} & = \Ind{X}{Y}{f(x)}.
\end{align*}
The third equality holds since if $b$ runs over all of $G$, then so does $a^{-1}b$ (that is, the change of variable $c = a^{-1}b$ is permitted). Thus, $\Ind{X}{Y}{f}\in\LG{G}{X}$. 
\end{proof}

The following theorem asserts that the linear transformations $\Res{X}{Y}{}$ and $\Ind{X}{Y}{}$ are Hermitian adjoint with respect to the Hermitian inner product defined earlier. This is a group-action version of Frobenius reciprocity.

\begin{theorem}[Frobenius reciprocity]
Let $G$ be a finite group, let $X$ be a finite $G$-set, and let $Y$ be an invariant subset of $X$. Then
\begin{equation}\label{eqn: Frobenius reciprocity}
    \gen{\Ind{X}{Y}{f}, g} = \gen{f, \Res{X}{Y}{g}}
\end{equation}
for all $f\in\LG{G}{Y}, g\in\LG{G}{X}$.
\end{theorem}
\begin{proof}
Direct computation shows that
\begin{align*}
\gen{\Ind{X}{Y}{f}, g} & = \dfrac{1}{\abs{X}}\lsum{x\in X}{}\Ind{X}{Y}{f(x)}\lbar{g(x)}\\
& = \dfrac{1}{\abs{X}}\lsum{x\in X}{}\left( \dfrac{\abs{X}}{\abs{G}\abs{Y}}\lsum{b\in G}{}\tilde{f}(b^{-1}\cdot x)\right)\lbar{g(x)}\\
& = \dfrac{1}{\abs{G}\abs{Y}}\lsum{x\in X}{}\left( \lsum{b\in G}{}\tilde{f}(b^{-1}\cdot x)\right)\lbar{g(x)} \\
& = \dfrac{1}{\abs{G}\abs{Y}}\lsum{x\in Y}{}\left( \lsum{b\in G}{}f(b^{-1}\cdot x)\right)\lbar{g(x)} \\
& = \dfrac{1}{\abs{G}\abs{Y}}\lsum{x\in Y}{} \abs{G}f(x)\lbar{g(x)} \\
& = \dfrac{1}{\abs{Y}}\lsum{x\in Y}{} f(x)\lbar{g(x)} \\
& = \dfrac{1}{\abs{Y}}\lsum{x\in Y}{}f(x)\lbar{\Res{X}{Y}{g(x)}}\\
& = \gen{f, \Res{X}{Y}{g}}.
\end{align*}
The fifth equality holds since $f\in \LG{G}{Y}$, which implies that $f(b^{-1}\cdot x)=f(x)$ for all $b\in G, x\in Y$.
\end{proof}

We remark that the inner product used on the right hand side of \eqref{eqn: Frobenius reciprocity} is computed by the same formula as in \eqref{eqn: inner product on L(X)} with $Y$ in place of $X$. This makes sense because if $Y$ is an invariant subset of $X$, then the $G$-action on $X$ restricts to the $G$-action on $Y$. In other words, the restriction of the Hermitian inner product of $L(X)$ to $L(Y)$ does define an inner product on $L(Y)$.

\hspace{0.3cm}

\par\noindent\textbf{Acknowledgment.} This work was supported by the Research Center in Mathematics and Applied Mathematics, Chiang Mai University. 

\bibliographystyle{amsplain}\addcontentsline{toc}{section}{References}
\bibliography{References}
\end{document}